\documentclass{article}

\usepackage{amsthm,amssymb,amsmath}
\usepackage{a4wide}

\newtheorem{theorem}{Theorem}
\newtheorem{corollary}[theorem]{Corollary}
\newtheorem{proposition}[theorem]{Proposition}
\newtheorem{definition}[theorem]{Definition}
\newtheorem{example}[theorem]{Example}
\newtheorem{remark}[theorem]{Remark}

\begin{document}

\title{A Fractional Calculus on Arbitrary Time Scales:\\
Fractional Differentiation and Fractional Integration\thanks{Part of first author's Ph.D.,
which is carried out at Sidi Bel Abbes University, Algeria.} \thanks{This is a preprint of 
a paper whose final and definite form will appear in the international journal \emph{Signal Processing}, 
ISSN 0165-1684. Paper submitted 04/Jan/2014; revised 19/Apr/2014; accepted for publication 12/May/2014.}} 

\author{Nadia Benkhettou$^1$\\ \texttt{benkhettou$_{-}$na@yahoo.fr}
\and Artur M. C. Brito da Cruz$^{2, 3}$\\ \texttt{artur.cruz@estsetubal.ips.pt}
\and Delfim F. M. Torres$^3$\thanks{Corresponding author.
Tel: +351 234370668; Fax: +351 234370066;  Email: delfim@ua.pt}\\ \texttt{delfim@ua.pt}}

\date{$^1$Laboratoire de Math\'{e}matiques, Universit\'{e} de Sidi Bel-Abb\`{e}s\\
B.P. 89, 22000 Sidi Bel-Abb\`{e}s, Algerie\\[0.3cm]
$^2$Escola Superior de Tecnologia de Set\'{u}bal\\
Estefanilha, 2910-761 Set\'{u}bal, Portugal\\[0.3cm]
$^3$\text{Center for Research and Development in Mathematics and Applications (CIDMA)}\\
Department of Mathematics, University of Aveiro, 3810-193 Aveiro, Portugal}

\maketitle

% ------------------------

\begin{abstract}
We introduce a general notion of fractional (noninteger) derivative
for functions defined on arbitrary time scales.
The basic tools for the time-scale fractional calculus
(fractional differentiation and fractional integration)
are then developed. As particular cases, one obtains
the usual time-scale Hilger derivative
when the order of differentiation is one, and a local
approach to fractional calculus when the time scale
is chosen to be the set of real numbers.

\bigskip


\bigskip

\end{abstract}

% ------------------------

\section{Introduction}

Fractional calculus refers to differentiation and integration
of an arbitrary (noninteger) order. The theory goes back to
mathematicians as Leibniz (1646--1716), Liouville (1809--1882),
Riemann (1826--1866), Letnikov (1837--1888), and Gr\"{u}nwald (1838--1920)
\cite{book:Kilbas,book:Samko}. During the last two decades, fractional calculus
has increasingly attracted the attention of researchers of many different fields
\cite{book:Benchohra,MR2870885,Baleanu:Nigmatullin,TM:K:M,book:FCV,MR2090004,book:Ortigueira,Yang:Baleanu:etal}.

Several definitions of fractional derivatives/integrals have
been defined in the literature, including those of Riemann--Liouville,
Gr\"{u}nwald--Letnikov, Hadamard, Riesz, Weyl and Caputo
\cite{book:Kilbas,Ortigueira:Trujillo:2012,book:Samko}.
In 1996, Kolwankar and Gangal proposed a local fractional derivative operator that applies
to highly irregular and nowhere differentiable
Weierstrass functions \cite{MR1911751,K:G:96}.
Here we introduce the notion of fractional derivative
on an arbitrary time scale $\mathbb{T}$
(cf. Definition~\ref{def:fd:ts}). In the particular case
$\mathbb{T} = \mathbb{R}$, one gets the local Kolwankar--Gangal
fractional derivative
$\lim_{h \rightarrow 0} \frac{f(t+h) - f(t)}{h^\alpha}$,
which has been considered in \cite{K:G:96,K:G:97} as the point
of departure for fractional calculus. One of the motivations
to consider such local fractional derivatives is
the possibility to deal with irregular signals,
so common in applications of signal processing \cite{K:G:97}.

A time scale is a model of time.
The calculus on time scales was initiated by Aulbach and Hilger in 1988
\cite{MR1062633}, in order to unify and generalize continuous
and discrete analysis \cite{H2,H1}.
It has a tremendous potential for applications and has recently
received much attention \cite{ABRP,BP,BP1,CK,MyID:252}.
The idea to join the two subjects --- the fractional calculus and
the calculus on time scales --- and to develop a
\emph{Fractional Calculus on Time Scales},
was born with the Ph.D. thesis of Bastos \cite{PhD:Nuno}.
See also \cite{MR2933070,PhD:Auch,Nuno:Z,Nuno:hZ,Nuno:Lap,Kisela,Rib:Ant,PhD:Williams}
and references therein. Here we introduce a general fractional calculus
on time scales and develop some of its basic properties.

Fractional calculus is of increasing importance in signal processing \cite{book:Ortigueira}.
This can be explained by several factors, such as the presence of internal noises
in the structural definition of the signals. Our fractional derivative
depends on the graininess function of the time scale.
We trust that this possibility can be very useful in applications of signal processing,
providing a concept of coarse-graining in time that can be used to model white noise
that occurs in signal processing or to obtain generalized entropies and new practical
meanings in signal processing. Indeed, let $\mathbb{T}$ be a time scale
(continuous time $\mathbb{T} = \mathbb{R}$, discrete time $\mathbb{T} = h \mathbb{Z}$, $h > 0$,
or, more generally, any closed subset of the real numbers, like the Cantor set).
Our results provide a mathematical framework to deal with functions/signals
$f(t)$ in signal processing that are not differentiable in the time scale,
that is, signals $f(t)$ for which the equality $\Delta f(t) = f^\Delta(t) \Delta t$ does not hold.
More precisely, we are able to model signal processes for which
$\Delta f(t) = f^{(\alpha)}(t) (\Delta t)^\alpha$,
$0 < \alpha \le 1$.

The time-scale calculus can be used to unify discrete and continuous
approaches to signal processing in one unique setting. Interesting in applications,
is the possibility to deal with more complex time domains. One extreme case, covered
by the theory of time scales and surprisingly relevant also for the process of signals,
appears when one fix the time scale to be the Cantor set
\cite{Baleanu:TM:etal,Yang:Srivastava:etal}.
The application of the local fractional derivative
in a time scale different from the classical time scales
$\mathbb{T} = \mathbb{R}$ and $\mathbb{T} = h \mathbb{Z}$ was proposed by
Kolwankar and Gangal themselves: see \cite{K:G:97,K:G:98}
where nondifferentiable signals defined on the Cantor set are considered.

The article is organized as follows. In Section~\ref{sec:prelim}
we recall the main concepts and tools necessary in the sequel.
Our results are given in Section~\ref{sec:MR}:
in Section~\ref{sub:sec:FD} the notion
of fractional derivative for functions defined
on arbitrary time scales is introduced
and the respective fractional
differential calculus developed;
the notion of fractional integral on time scales, and some of its
basic properties, is investigated in Section~\ref{sub:sec:FI}.
We end with Section~\ref{sec:Conc} of conclusions and future work.

% ------------------------

\section{Preliminaries}
\label{sec:prelim}

A time scale $ \mathbb{T}$ is an arbitrary
nonempty closed subset of $ \mathbb{R}$.
Here we only recall the necessary concepts of the calculus
on time scales. The reader interested on the subject is referred
to the books \cite{BP,BP1}. For a good survey see \cite{ABRP}.

\begin{definition}
\label{def:jump:oper}
Let $\mathbb{T}$ be a time scale. For $t \in \mathbb{T}$
we define the forward jump operator
$\sigma:\mathbb{T}\rightarrow \mathbb{T}$ by
$\sigma(t):=\inf\{s \in\mathbb{T} : s > t\}$,
and the backward jump operator
$\rho:\mathbb{T}\rightarrow \mathbb{T}$ by
$\rho(t):=\sup\{s \in\mathbb{T} : s < t\}$.
\end{definition}

\begin{remark}
In Definition~\ref{def:jump:oper}, we put $\inf \emptyset =\sup \mathbb{T}$
(i.e., $\sigma(t)= t$) if $\mathbb{T}$ has a maximum $t$,
and $\sup \emptyset =\inf \mathbb{T}$ (i.e., $\rho(t)= t$)
if $\mathbb{T}$ has a minimum $t$, where $\emptyset$ denotes
the empty set.
\end{remark}

If $\sigma(t) > t$, then we say that $t$ is right-scattered;
if $\rho(t) < t$, then $t$ is said to be left-scattered.
Points that are simultaneously right-scattered and left-scattered
are called isolated. If $t < \sup\mathbb{T}$ and $\sigma(t) = t$,
then $t$ is called right-dense; if $t >\inf \mathbb{T}$ and $\rho(t)= t$,
then $t$ is called left-dense. The graininess function
$\mu :\mathbb{T}\rightarrow [0,\infty)$ is defined by
$\mu(t) :=\sigma(t) - t$.

We make use of the set $\mathbb{T}^{\kappa}$,
which is derived from the time scale $\mathbb{T}$ as follows: if $\mathbb{T}$
has a left-scattered maximum $M$, then $\mathbb{T}^{\kappa}=\mathbb{T} \setminus \{M\}$;
otherwise, $\mathbb{T}^{\kappa}=\mathbb{T}$.

\begin{definition}[Delta derivative \cite{AB}]
Assume $f:\mathbb{T}\rightarrow \mathbb{R}$ and let
$t\in \mathbb{T}^{\kappa}$. We define
$$
f^{\Delta}(t)=\lim_{s\rightarrow t}\frac{f(\sigma(s))-f(t)}{\sigma(s)-t},
\quad t \neq \sigma(s),
$$
provided the limit exists. We call $f^{\Delta}(t)$ the delta derivative
(or Hilger derivative) of $f$ at $t$. Moreover, we say that $f$
is delta differentiable on $\mathbb{T}^{\kappa}$ provided
$f^{\Delta}(t)$ exists for all $t\in \mathbb{T}^{\kappa}$. The function
$f^{\Delta}:\mathbb{T}^{\kappa}\rightarrow \mathbb{R}$ is then called
the delta derivative of $f$ on $\mathbb{T}^{\kappa}$.
\end{definition}

Delta derivatives of higher-order are defined in
the usual way. Let $r\in\mathbb{N}$,
$\mathbb{T}^{\kappa^{0}} := \mathbb{T}$, and
$\mathbb{T}^{\kappa^i}:=\left(\mathbb{T}^{\kappa^{i-1}}\right)^\kappa$,
$i = 1, \ldots, r$. For convenience we also put $f^{\Delta^0} = f$
and $f^{\Delta^1} = f^\Delta$. The $r$th-delta derivative
$f^{\Delta^r}$ is given by $f^{\Delta^r} =
\left(f^{\Delta^{r-1}}\right)^\Delta: \mathbb{T}^{\kappa^r} \rightarrow \mathbb{R}$
provided $f^{\Delta^{r-1}}$ is delta differentiable.

The following notions will be useful in connection with
the fractional integral (Section~\ref{sub:sec:FI}).

\begin{definition}
A function $f:\mathbb{T} \rightarrow \mathbb{R}$ is called regulated
provided its right-sided limit exist (finite) at all right-dense points
in $\mathbb{T}$ and its left-sided limits exist (finite)
at all left-dense points in $\mathbb{T}$.
\end{definition}

\begin{definition}
A function $f:\mathbb{T}\rightarrow \mathbb{R}$ is called rd-continuous provided
it is continuous at right-dense points in $\mathbb{T} $ and its left-sided limits
exist (finite) at left-dense points in $\mathbb{T}$. The set of rd-continuous
functions $f:\mathbb{T}\rightarrow \mathbb{R}$ is denoted by $\mathcal{C}_{rd}$.
\end{definition}

% ------------------------

\section{Main Results}
\label{sec:MR}

We develop the basic tools of any fractional calculus:
fractional differentiation (Section~\ref{sub:sec:FD})
and fractional integration (Section~\ref{sub:sec:FI}).
Let $\mathbb{T}$ be a time scale, $t\in \mathbb{T}$,
and $\delta >0$. We define the left $\delta$-neighborhood of $t$
as $\mathcal{U}^{-} :=\left] t-\delta ,t\right[ \cap \mathbb{T}$.

% ------------------------

\subsection{Fractional Differentiation}
\label{sub:sec:FD}

We begin by introducing a new notion:
the fractional derivative of order $\alpha \in ]0,1]$
for functions defined on arbitrary time scales.
For $\alpha = 1$ we obtain the usual delta derivative
of the time-scale calculus.

\begin{definition}
\label{def:fd:ts}
Let $f:\mathbb{T}\rightarrow \mathbb{R}$, $t\in \mathbb{T}^{\kappa }$,
and $\alpha \in ]0,1]$. For $\alpha \in ]0,1]\cap \left\{ 1/q
: q \text{ is a odd number}\right\}$ (resp. $\alpha \in ]0,1] \setminus
\left\{ 1/q : q\text{ is a odd number}\right\}$) we define
$f^{(\alpha )}(t)$ to be the number (provided it exists) with the property
that, given any $\epsilon >0$, there is a $\delta$-neighborhood $\mathcal{U}\subset
\mathbb{T}$ of $t$ (resp. left $\delta$-neighborhood
$\mathcal{U}^{-}\subset \mathbb{T}$ of $t$), $\delta > 0$, such that
\begin{equation*}
\left \vert \left[ f(\sigma (t))-f(s)\right] -f^{(\alpha )}(t)\left[ \sigma
(t)-s\right] ^{\alpha }\right \vert \leq \epsilon \left \vert \sigma
(t)-s\right \vert^{\alpha}
\end{equation*}
for all $s\in \mathcal{U}$ (resp. $s\in \mathcal{U}^{-}$). We call
$f^{(\alpha )}(t)$ the fractional derivative of $f$ of order $\alpha $ at $t$.
\end{definition}

Along the text $\alpha$ is a real number in the interval $]0,1]$.
The next theorem provides some useful relationships
concerning the fractional derivative on time scales
introduced in Definition~\ref{def:fd:ts}.

\begin{theorem}
\label{T1}
Assume $f:\mathbb{T}\rightarrow \mathbb{R}$ and let
$t\in \mathbb{T}^{\kappa }$. The following properties hold:
\begin{description}
\item[(i)] Let $\alpha \in ]0,1]\cap \left\{ \frac{1}{q} : q\text{ is a odd
number}\right\} $. If $t$ is right-dense and if $f$ is fractional
differentiable of order $\alpha$ at $t$, then $f$ is continuous at $t$.

\item[(ii)] Let $\alpha \in ]0,1] \setminus \left\{\frac{1}{q} : q\text{ is
a odd number}\right\}$. If $t$ is right-dense and if $f$ is fractional
differentiable of order $\alpha$ at $t$, then $f$ is left-continuous at $t$.

\item[(iii)] If $f$ is continuous at $t$ and $t$ is right-scattered, then $f$
is fractional differentiable of order $\alpha$ at $t$ with
\begin{equation*}
f^{(\alpha )}(t)=\frac{f^{\sigma }(t)-f(t)}{(\mu (t))^{\alpha }}.
\end{equation*}

\item[(iv)] Let $\alpha \in ]0,1]\cap \left\{ \frac{1}{q} : q\text{ is a odd
number}\right\} $. If $t$ is right-dense, then $f$ is fractional
differentiable of order $\alpha$ at $t$ if, and only if, the limit
\begin{equation*}
\lim_{s\rightarrow t}\frac{f(t)-f(s)}{(t-s)^{\alpha}}
\end{equation*}
exists as a finite number. In this case,
\begin{equation*}
f^{(\alpha )}(t)=\lim_{s\rightarrow t}\frac{f(t)-f(s)}{(t-s)^{\alpha}}.
\end{equation*}

\item[(v)] Let $\alpha \in ]0,1] \setminus \left\{ \frac{1}{q} : q\text{ is a
odd number}\right\}$. If $t$ is right-dense, then $f$ is fractional
differentiable of order $\alpha$ at $t$ if, and only if, the limit
\begin{equation*}
\lim_{s\rightarrow t^{-}}\frac{f(t)-f(s)}{(t-s)^{\alpha}}
\end{equation*}
exists as a finite number. In this case,
\begin{equation*}
f^{(\alpha )}(t)=\lim_{s\rightarrow t^{-}}\frac{f(t)-f(s)}{(t-s)^{\alpha }}.
\end{equation*}

\item[(vi)] If $f$ is fractional differentiable of order $\alpha$ at $t$,
then $f(\sigma (t))=f(t)+(\mu (t))^{\alpha }f^{(\alpha )}(t)$.
\end{description}
\end{theorem}

\begin{proof}
$(i)$ Assume that $f$ is fractional differentiable at $t$. Then, there
exists a neighborhood $\mathcal{U}$ of $t$ such that
\begin{equation*}
\left \vert \left[ f(\sigma (t))-f(s)\right] -f^{(\alpha )}(t)\left[ \sigma
(t)-s\right] ^{\alpha }\right \vert \leq \epsilon \left \vert \sigma
(t)-s\right \vert^{\alpha}
\end{equation*}
for $s\in \mathcal{U}$. Therefore, for all $s \in \mathcal{U} \cap
\left]t-\epsilon ,t+\epsilon \right[$,
\begin{multline*}
\left \vert f\left( t\right) -f\left( s\right) \right \vert  \leq \left\vert
\left[ f^{\sigma }(t)-f(s)\right] -f^{(\alpha )}(t)\left[ \sigma (t)
-s\right]^{\alpha}\right \vert\\
+\left \vert \left[ f^{\sigma }(t)-f(t)\right]
-f^{(\alpha )}(t)\left[ \sigma (t)-t\right] ^{\alpha }\right \vert
+\left \vert f^{(\alpha )}(t)\right \vert \left \vert \left[ \sigma (t)-s
\right] ^{\alpha }-\left[ \sigma (t)-t\right] ^{\alpha }\right \vert
\end{multline*}
and, since $t$ is a right-dense point,
\begin{equation*}
\begin{split}
\left \vert f\left( t\right) -f\left( s\right) \right \vert
&\leq \left \vert
\left[ f^{\sigma }(t)-f(s)\right] -f^{(\alpha )}(t)\left[ \sigma (t)
-s\right]^{\alpha }\right \vert +\left \vert f^{(\alpha )}(t)\left[
t-s\right]^{\alpha}\right\vert \\
&\leq \epsilon \left \vert t-s\right \vert^{\alpha}
+\left \vert f^{(\alpha)}(t)\right \vert \left \vert t-s\right \vert^{\alpha}\\
&\leq \epsilon ^{\alpha }\left[ \epsilon +\left\vert f^{(\alpha)}(t)
\right\vert \right].
\end{split}
\end{equation*}
It follows the continuity of $f$ at $t$.

$(ii)$ The proof is similar to the proof of $(i)$,
where instead of considering the neighborhood $\mathcal{U}$ of $t$
we consider a left neighborhood $\mathcal{U}^{-}$ of $t$.

$(iii)$ Assume that $f$ is continuous at $t$ and $t$ is right-scattered.
By continuity,
\begin{equation*}
\lim_{s\rightarrow t}\frac{f^{\sigma }(t)-f(s)}{(\sigma (t)-s)^{\alpha }}
=\frac{f^{\sigma }(t)-f(t)}{(\sigma (t)-t)^{\alpha}}
=\frac{f^{\sigma}(t)-f(t)}{(\mu (t))^{\alpha }}.
\end{equation*}
Hence, given $\epsilon >0$ and $\alpha \in ]0,1] \cap \left\{ 1/q : q
\text{ is a odd number}\right\}$, there is a neighborhood $\mathcal{U}$ of
$t$ (or $\mathcal{U}^{-}$ if $\alpha \in ]0,1] \setminus \left\{1/q
: q\text{ is a odd number}\right\}$) such that
\begin{equation*}
\left \vert \frac{f^{\sigma }(t)-f(s)}{(\sigma (t)-s)^{\alpha }}
-\frac{f^{\sigma }(t)-f(t)}{(\mu (t))^{\alpha }}\right \vert
\leq \epsilon
\end{equation*}
for all $s\in \mathcal{U}$ (resp. $\mathcal{U}^{-}$). It follows that
\begin{equation*}
\left \vert \left[ f^{\sigma }(t)-f(s)\right] -\frac{f^{\sigma }(t)-f(t)}{
(\mu (t))^{\alpha }}(\sigma (t)-s)^{\alpha }\right \vert \leq \epsilon
|\sigma (t)-s|^{\alpha}
\end{equation*}
for all $s\in \mathcal{U}$ (resp. $\mathcal{U}^{-}$). Hence, we get the
desired result:
\begin{equation*}
f^{(\alpha )}(t)=\frac{f^{\sigma }(t)-f(t)}{(\mu (t))^{\alpha}}.
\end{equation*}

$(iv)$ Assume that $f$ is fractional differentiable of order $\alpha $ at $t$
and $t$ is right-dense. Let $\epsilon > 0$ be given. Since $f$ is fractional
differentiable of order $\alpha $ at $t$, there is a neighborhood
$\mathcal{U}$ of $t$ such that
\begin{equation*}
\left \vert \lbrack f^{\sigma }(t)-f(s)]-f^{(\alpha )}(t)(\sigma
(t)-s)^{\alpha }\right \vert \leq \epsilon |\sigma (t)-s|^{\alpha}
\end{equation*}
for all $s\in \mathcal{U}$. Since $\sigma (t)=t$,
\begin{equation*}
\left \vert \lbrack f(t)-f(s)]-f^{(\alpha )}(t)(t-s)^{\alpha }\right \vert
\leq \epsilon |t-s|^{\alpha}
\end{equation*}
for all $s\in \mathcal{U}$. It follows that
\begin{equation*}
\left \vert \frac{f(t)-f(s)}{(t-s)^{\alpha }}-f^{(\alpha )}(t)\right \vert
\leq \epsilon
\end{equation*}
for all $s\in \mathcal{U}$, $s\neq t$. Therefore, we get the desired result:
\begin{equation*}
f^{(\alpha )}(t)=\lim_{s\rightarrow t}\frac{f(t)-f(s)}{(t-s)^{\alpha }}.
\end{equation*}
Now assume that
\begin{equation*}
\lim_{s\rightarrow t}\frac{f(t)-f(s)}{(t-s)^{\alpha}}
\end{equation*}
exists and is equal to $L$ and $t$ is right-dense.
Then, there exists $\mathcal{U}$ such that
\begin{equation*}
\left \vert \frac{f(t)-f(s)}{(t-s)^{\alpha }}-L\right \vert \leq \epsilon
\end{equation*}
for all $s\in \mathcal{U}$. Because $t$ is right-dense,
\begin{equation*}
\left \vert \frac{f^{\sigma }(t)-f(s)}{(\sigma \left( t\right) -s)^{\alpha}}
-L\right \vert \leq \epsilon.
\end{equation*}
Therefore,
\begin{equation*}
\left \vert \left[ f^{\sigma }(t)-f(s)\right]
-L\left( \sigma (t)-s\right)^{\alpha }\right\vert
\leq \epsilon |\sigma \left( t\right) -s|^{\alpha},
\end{equation*}
which lead us to the conclusion that $f$ is fractional differentiable of
order $\alpha $ at $t$ and $f^{(\alpha )}(t)=L$.

$(v)$ The proof is similar to the
proof of $(iv)$, where instead of considering the neighborhood
$\mathcal{U}$ of $t$ we consider a left-neighborhood $\mathcal{U}^{-}$ of $t$.

$(vi)$ If $\sigma (t)=t$, then $\mu (t)=0$ and
\begin{equation*}
f^{\sigma }(t))=f(t)=f(t)+(\mu (t))^{\alpha }f^{(\alpha )}(t).
\end{equation*}
On the other hand, if $\sigma (t)>t$, then by $(iii)$
\begin{equation*}
f^{\sigma }(t)=f(t)+(\mu (t))^{\alpha }\cdot \frac{f^{\sigma }(t)-f(t)}{(\mu
(t))^{\alpha }}=f(t)+(\mu (t))^{\alpha }f^{(\alpha )}(t).
\end{equation*}
The proof is complete.
\end{proof}

\begin{remark}
In a time scale $\mathbb{T}$, due to the inherited topology of
the real numbers, a function $f$ is always continuous at any isolated point $t$.
\end{remark}

\begin{proposition}
\label{E1:i}
If $f:\mathbb{T}\rightarrow \mathbb{R}$ is defined by $f(t)= c$
for all $t\in\mathbb{T}$, $c\in \mathbb{R}$,
then $f^{(\alpha)}(t)\equiv 0$.
\end{proposition}

\begin{proof}
If $t$ is right-scattered, then, by Theorem~\ref{T1} (iii), one has
$$
f^{(\alpha)}(t)=\frac{f(\sigma(t))-f(t)}{(\mu(t))^{\alpha}}
=\frac{c-c}{(\mu(t))^{\alpha}}=0.
$$
Assume $t$ is right-dense. Then,
by Theorem~\ref{T1} (iv) and (v), it follows that
$$
f^{(\alpha)}(t) = \lim_{s \rightarrow t}\frac{c-c}{(t-s)^{\alpha}} = 0.
$$
This concludes the proof.
\end{proof}

\begin{proposition}
\label{E1:ii}
If $f:\mathbb{T}\rightarrow \mathbb{R}$ is defined by $f(t)=t$
for all $t\in \mathbb{T}$, then
\begin{equation*}
f^{(\alpha )}(t)
=
\begin{cases}
(\mu (t))^{1-\alpha } & \textrm{ if } \alpha \neq 1, \\
1 & \textrm{ if } \alpha =1.
\end{cases}
\end{equation*}
\end{proposition}

\begin{proof}
From Theorem~\ref{T1} (vi) it follows that
$\sigma(t) = t + (\mu(t))^{\alpha} f^{(\alpha)}(t)$, that is,
$\mu(t) = (\mu(t))^{\alpha} f^{(\alpha)}(t)$.
If $\mu(t) \ne 0$, then $f^{(\alpha)}(t)=(\mu(t))^{1-\alpha}$
and the desired relation is proved. Assume now that
$\mu(t) = 0$, that is, $\sigma(t) = t$. In this case
$t$ is right-dense and by Theorem~\ref{T1} (iv) and (v) it follows that
$$
f^{(\alpha)}(t) = \lim_{s \rightarrow t}\frac{t-s}{(t-s)^{\alpha}}.
$$
Therefore, if $\alpha =1$, then $f^{(\alpha )}(t)=1$;
if $0<\alpha <1$, then $f^{(\alpha )}(t)=0$.
The proof is complete.
\end{proof}

Let us consider now the two classical cases $\mathbb{T}=\mathbb{R}$
and $\mathbb{T}= h \mathbb{Z}$, $h > 0$.

\begin{corollary}
Function $f :\mathbb{R} \rightarrow \mathbb{R}$ is fractional differentiable
of order $\alpha$ at point $t \in \mathbb{R}$ if, and only if, the limit
$$
\lim_{s\rightarrow t}\frac{f(t)-f(s)}{(t-s)^{\alpha}}
$$
exists as a finite number. In this case,
\begin{equation}
\label{KG:der}
f^{(\alpha)}(t)=\lim_{s\rightarrow t}\frac{f(t)-f(s)}{(t-s)^{\alpha}}.
\end{equation}
\end{corollary}

\begin{proof}
Here $\mathbb{T}=\mathbb{R}$ and all points are right-dense.
The result follows from Theorem~\ref{T1} (iv) and (v).
Note that if $\alpha \in ]0,1] \setminus
\left\{ \frac{1}{q}:q\text{ is a odd number}\right\}$,
then the limit only makes sense as a left-side limit.
\end{proof}

\begin{remark}
The definition \eqref{KG:der} corresponds to the well-known
Kolwankar--Gangal approach to fractional calculus \cite{K:G:96,Wang:FDA12}.
\end{remark}

\begin{corollary}
Let $h > 0$. If $f :h\mathbb{Z} \rightarrow \mathbb{R}$, then
$f$ is fractional differentiable of order $\alpha$ at $t\in h\mathbb{Z}$ with
$$
f^{(\alpha)}(t) =\frac{f(t+h)-f(t)}{h^\alpha}.
$$
\end{corollary}

\begin{proof}
Here $\mathbb{T}=h\mathbb{Z}$
and all points are right-scattered.
The result follows from Theorem~\ref{T1} (iii).
\end{proof}

We now give an example using a more sophisticated time scale: the Cantor set.

\begin{example}
Let $\mathbb{T}$ be the Cantor set. It is known (see Example~1.47 of \cite{BP})
that $\mathbb{T}$ does not contain any isolated point, and that
$$
\sigma(t) =
\begin{cases}
t + \frac{1}{3^{m+1}} & \text{ if } t \in L,\\
t  & \text{ if } t \in \mathbb{T} \setminus L,
\end{cases}
$$
where
$$
L = \left\{\sum_{k=1}^{m} \frac{a_k}{3^k} + \frac{1}{3^{m+1}} : m \in \mathbb{N} \text{ and }
a_k \in \{0, 2\} \text{ for all } 1 \le k \le m\right\}.
$$
Thus,
$$
\mu(t) =
\begin{cases}
\frac{1}{3^{m+1}} & \text{ if } t \in L,\\
0  & \text{ if } t \in \mathbb{T} \setminus L.
\end{cases}
$$
Let $f : \mathbb{T} \rightarrow \mathbb{R}$ be continuous and $\alpha \in ]0,1]$.
It follows from Theorem~\ref{T1} that the fractional derivative of order
$\alpha$ of a function $f$ defined on the Cantor set is given by
$$
f^{(\alpha)}(t) =
\begin{cases}
\left[f\left(t + \frac{1}{3^{m+1}}\right)-f(t)\right]3^{(m+1)\alpha} & \text{ if } t \in L,\\[0.3cm]
\displaystyle \lim_{s \rightsquigarrow t}  \frac{f(t)-f(s)}{(t-s)^\alpha}
& \text{ if } t \in \mathbb{T} \setminus L,
\end{cases}
$$
where $\lim_{s\rightsquigarrow t} = \lim_{s \rightarrow t}$ if $\alpha = \frac{1}{q}$ with $q$ an odd number,
and $\lim_{s\rightsquigarrow t} = \lim_{s \rightarrow t^{-}}$ otherwise.
\end{example}

For the fractional derivative on time scales to be useful,
we would like to know formulas for the
derivatives of sums, products and quotients
of fractional differentiable functions. This is
done according to the following theorem.

\begin{theorem}
\label{T2}
Assume  $f, g : \mathbb{T} \rightarrow \mathbb{R}$
are fractional differentiable of order $\alpha$ at
$t \in \mathbb{T}^{\kappa}$. Then,
\begin{description}
\item[(i)] the sum $f+g:\mathbb{T}\rightarrow \mathbb{R}$
is fractional differentiable at $t$ with
$(f+g)^{(\alpha)}(t)=f^{(\alpha)}(t)+g^{(\alpha)}(t)$;

\item[(ii)] for any constant $\lambda$, $\lambda f :\mathbb{T}\rightarrow \mathbb{R}$
is fractional differentiable at $t$ with
$(\lambda f)^{(\alpha)}(t)=\lambda f^{(\alpha)}(t)$;

\item[(iii)] if $f$ and $g$ are continuous, then
the product $f g :\mathbb{T}\rightarrow \mathbb{R}$
is fractional differentiable at $t$ with
\begin{equation*}
\begin{split}
(fg)^{(\alpha)}(t)
&=f^{(\alpha)}(t)g(t)+f(\sigma(t))g^{(\alpha)}(t)\\
&= f^{(\alpha)}(t)g(\sigma(t)) + f(t)g^{(\alpha)}(t);
\end{split}
\end{equation*}

\item[(iv)] if $f$ is continuous and $f(t)f(\sigma(t))\neq 0$,
then $\frac{1}{f}$ is fractional differentiable at $t$ with
$$
\left(\frac{1}{f}\right)^{(\alpha)}(t)
= -\frac{f^{(\alpha)}(t)}{f(t)f(\sigma(t))};
$$

\item[(v)] if $f$ and $g$ are continuous and $g(t)g(\sigma(t))\neq 0$,
then $\frac{f}{g}$ is fractional differentiable at $t$ with
$$
\left(\frac{f}{g}\right)^{(\alpha)}(t)
=\frac{f^{(\alpha)}(t)g(t)-f(t)g^{(\alpha)}(t)}{g(t)g(\sigma(t))}.
$$
\end{description}
\end{theorem}

\begin{proof}
Let us consider that $\alpha \in ]0,1]\cap \left\{ \frac{1}{q} : q
\text{ is a odd number}\right\}$. The proofs for the case
$\alpha \in ]0,1] \setminus \left\{\frac{1}{q} : q \text{ is a odd number}\right \}$
are similar: one just needs to choose the proper left-sided neighborhoods.
Assume that $f$ and $g$ are fractional differentiable at $t \in\mathbb{T}^{\kappa}$.
$(i)$ Let $\epsilon > 0$. Then there exist neighborhoods
$\mathcal{U}_{1}$ and $\mathcal{U}_{2}$ of $t$ for which
\begin{equation*}
\left|f(\sigma(t))-f(s)-f^{(\alpha)}(t)[\sigma(t)-s]^{\alpha}\right|
\leq \frac{\epsilon}{2}|\sigma(t)-s|^{\alpha}~~for ~all~~s\in \mathcal{U}_{1}
\end{equation*}
and
\begin{equation*}
\left|g(\sigma(t))-g(s)-g^{(\alpha)}(t)[\sigma(t)-s]^{\alpha}\right|
\leq \frac{\epsilon}{2}|\sigma(t)-s|^{\alpha}~~for ~all~~s\in \mathcal{U}_{2}.
\end{equation*}
Let $\mathcal{U}=\mathcal{U}_{1}\cap \mathcal{U}_{2}$. Then
\begin{equation*}
\begin{split}
\biggl|(f&+g)(\sigma(t))-(f+g)(s)-\left[f^{(\alpha)}(t)
+g^{(\alpha)}(t)\right](\sigma(t)-s)^{\alpha}\biggr|\\
&=\left|f(\sigma(t))-f(s)-f^{(\alpha)}(t)[\sigma(t)-s]^{\alpha}
+g(\sigma(t))-g(s)-g^{(\alpha)}(t)[\sigma(t)-s]^{\alpha}\right|\\
&\leq \left|f(\sigma(t))-f(s)-f^{(\alpha)}(t)[\sigma(t)-s]^{\alpha}\right|
+\left|g(\sigma(t))-g(s)-g^{(\alpha)}(t)[\sigma(t)-s]^{\alpha}\right|\\
&\leq \frac{\epsilon}{2}|\sigma(t)-s|^{\alpha}+\frac{\epsilon}{2}|\sigma(t)-s|^{\alpha}
=\epsilon |\sigma(t)-s|^{\alpha}
\end{split}
\end{equation*}
for all $s\in \mathcal{U}$. Therefore, $f+g$ is fractional differentiable at $t$ and
$(f+g)^{(\alpha)}(t)=f^{\alpha}(t)+g^{(\alpha)}(t)$.
$(ii)$ Let $\epsilon > 0$. Then there exists
a neighborhood $\mathcal{U}$ of $t$ with
\begin{equation*}
\left|f(\sigma(t))-f(s)-f^{(\alpha)}(t)[\sigma(t)-s]^{\alpha}\right|
\leq \epsilon|\sigma(t)-s|^{\alpha} \text{ for all } s\in \mathcal{U}.
\end{equation*}
It follows that
\begin{equation*}
\left|(\lambda f)(\sigma(t))-(\lambda f)(s)
-\lambda f^{(\alpha)}(t)[\sigma(t)-s]^{\alpha}\right|
\leq  \epsilon |\lambda| \, |\sigma(t)-s|^{\alpha} \text{ for  all } s \in \mathcal{U}.
\end{equation*}
Therefore, $\lambda f$ is fractional differentiable at $t$ and
$(\lambda f)^{\alpha}=\lambda f^{(\alpha)}$ holds at $t$.
$(iii)$ If $t$ is right-dense, then
\begin{equation*}
\begin{split}
(fg)^{(\alpha )}(t)
&=\lim_{s\rightarrow t}\frac{\left( fg\right)
(t)-\left( fg\right) (s)}{(t-s)^{\alpha }} \\
&=\lim_{s\rightarrow t}\frac{f(t)-f(s)}{(t-s)^{\alpha }}g\left( t\right)
+\lim_{s\rightarrow t}\frac{g(t)-g(s)}{(t-s)^{\alpha }}f\left( s\right)\\
&= f^{(\alpha )}(t)g(t)+g^{(\alpha )}(t)f(t) \\
&= f^{(\alpha )}(t)g(t)+f(\sigma (t))g^{(\alpha )}(t).
\end{split}
\end{equation*}
If $t$ is right-scattered, then
\begin{equation*}
\begin{split}
\left( fg\right)^{(\alpha )}(t)
&= \frac{\left( fg\right)^{\sigma}(t)
-\left( fg\right) (t)}{(\mu (t))^{\alpha }} \\
&=\frac{f^{\sigma }(t)-f(t)}{(\mu (t))^{\alpha }}g\left( t\right) +\frac{
g^{\sigma }(t)-g(t)}{(\mu (t))^{\alpha }}f^{\sigma }(t)\\
&=f^{(\alpha )}(t)g(t)+f(\sigma (t))g^{(\alpha )}(t).
\end{split}
\end{equation*}
The other product rule formula follows by interchanging in
$\left( fg\right)^{(\alpha )}(t)=f^{(\alpha )}(t)g(t)+f(\sigma (t))g^{(\alpha )}(t)$
the functions $f$ and $g$.
$(iv)$ We use the fractional derivative of a constant (Proposition~\ref{E1:i})
and Theorem~\ref{T2} $(iii)$ just proved: from Proposition~\ref{E1:i} we know that
$$
\left(f \cdot \frac{1}{f}\right)^{(\alpha)}(t)=(1)^{(\alpha)}(t)=0
$$
and, therefore, by (iii)
$$
\left(\frac{1}{f}\right)^{(\alpha)}(t)f(\sigma(t))
+f^{(\alpha)}(t)\frac{1}{f(t)}=0.
$$
Since we are assuming $f(\sigma(t))\neq 0$,
\begin{equation*}
\left(\frac{1}{f}\right)^{(\alpha)}(t)
=-\frac{f^{(\alpha)}(t)}{f(t)f(\sigma(t))}.
\end{equation*}
For the quotient formula $(v)$, we use $(ii)$ and $(iv)$ to calculate
\begin{equation*}
\begin{split}
\left(\frac{f}{g}\right)^{(\alpha)}(t)&=\left(f \cdot \frac{1}{g}\right)^{(\alpha)}(t)\\
&=f(t)\left(\frac{1}{g}\right)^{(\alpha)}(t)+f^{(\alpha)}(t)\frac{1}{g(\sigma(t))}\\
&=-f(t)\frac{g^{(\alpha)}(t)}{g(t)g(\sigma(t))}+f^{(\alpha)}(t)\frac{1}{g(\sigma(t))}\\
&=\frac{f^{(\alpha)}(t)g(t)-f(t)g^{(\alpha)}(t)}{g(t)g(\sigma(t))}.
\end{split}
\end{equation*}
This concludes the proof.
\end{proof}

The following theorem is proved in \cite{BP} for $\alpha = 1$.
Here we show its validity for $\alpha \in \left] 0,1\right[$.

\begin{theorem}
\label{thm:der:pf}
Let $c$ be a constant, $m \in \mathbb{N}$, and $\alpha \in \left] 0,1\right[$.
\begin{description}
\item[(i)] If $f(t) =(t-c)^{m}$, then
$$
f^{(\alpha)}(t)=(\mu(t))^{1-\alpha}
\sum_{\nu = 0}^{m-1}\left(\sigma(t)-c\right)^{\nu}(t-c)^{m-1-\nu}.
$$
\item[(ii)] If $g(t)=\frac{1}{(t-c)^{m}}$, then
$$
g^{(\alpha)}(t)=-(\mu(t))^{1-\alpha}
\sum_{\nu = 0}^{m-1}\frac{1}{(\sigma(t)-c)^{m-\nu}(t-c)^{\nu+1}},
$$
provided $(t-c)\left(\sigma(t)-c\right) \neq 0$.
\end{description}
\end{theorem}

\begin{proof}
We prove the first formula by induction.
If $ m=1$, then $f(t)=t-c$ and $f^{(\alpha)}(t)=(\mu(t))^{1-\alpha}$
holds from Propositions~\ref{E1:i} and \ref{E1:ii} and Theorem~\ref{T2} $(i)$.
Now assume that
$$
f^{(\alpha)}(t)=(\mu(t))^{1-\alpha}
\sum_{\nu = 0}^{m-1}(\sigma(t)-c)^{\nu}(t-c)^{m-1-\nu}
$$
holds for $f(t) =(t-c)^{m}$ and let $F(t)=(t-c)^{m+1}=(t-c)f(t)$.
We use the product rule (Theorem~\ref{T2} $(iii)$) to obtain
\begin{equation*}
\begin{split}
F^{(\alpha)}(t)&=(t-c)^{(\alpha)}f(\sigma(t))+f^{(\alpha)}(t)(t-c)
=(\mu(t))^{1-\alpha}f(\sigma(t))+f^{(\alpha)}(t)(t-c)\\
&=(\mu(t))^{1-\alpha}(\sigma(t)-c)^{m}+(\mu(t))^{1-\alpha}(t)(t-c)
\sum_{\nu = 0}^{m-1}(\sigma(t)-c)^{\nu}(t-c)^{m-1-\nu}\\
&=(\mu(t))^{1-\alpha}\left[( \sigma(t)-c)^{m}
+ \sum_{\nu = 0}^{m-1}(\sigma(t)-c)^{\nu}(t-c)^{m-\nu}\right]\\
&=(\mu(t))^{1-\alpha} \sum_{\nu = 0}^{m}(\sigma(t)-c)^{\nu}(t-c)^{m-\nu}.
\end{split}
\end{equation*}
Hence, by mathematical induction, part $(i)$ holds.
For $g(t)=\frac{1}{(t-c)^{m}}=\frac{1}{f(t)}$,
we apply Theorem~\ref{T2} $(iv)$ to obtain
\begin{equation*}
\begin{split}
g^{(\alpha)}(t)&=-\frac{f^{(\alpha)}(t)}{f(t)f(\sigma(t))}
=-(\mu(t))^{1-\alpha}\frac{\sum_{\nu = 0}^{m-1}(\sigma(t)
-c)^{\nu}(t-c)^{m-1-\nu}}{(t-c)^{m}(\sigma(t)-c)^{m}}\\
&=-(\mu(t))^{1-\alpha}\sum_{\nu = 0}^{m-1}
\frac{1}{(t-c)^{\nu+1}(\sigma(t)-c)^{m-\nu}},
\end{split}
\end{equation*}
provided $(t-c)\left(\sigma(t)-c\right) \neq 0$.
\end{proof}

Let us illustrate Theorem~\ref{thm:der:pf} in special cases.

\begin{example}
\label{ex:17}
Let $\alpha \in \left]0,1\right[$.
\begin{description}
\item[(i)] If $f(t)=t^{2}$,
then $f^{(\alpha)}(t)=(\mu(t))^{1-\alpha} [\sigma(t)+t]$.

\item[(ii)] If $f(t)=t^{3}$,
then $f^{(\alpha)}(t)=(\mu(t))^{1-\alpha} [t^{2}+t\sigma(t)+(\sigma(t))^{2}]$.

\item[(iii)] If $f(t)=\frac{1}{t}$,
then $f^{(\alpha)}(t)= -\frac{(\mu(t))^{1-\alpha}}{t\sigma(t)}$.
\end{description}
\end{example}

From the results already obtained, it is not difficult to see
that the fractional derivative does not satisfy a chain rule
like $(f\circ g)^{(\alpha)}(t)=f^{(\alpha)}(g(t)) g^{(\alpha)}(t)$:

\begin{example}
\label{ex:conterex:cr}
Let $\alpha \in \left]0,1\right[$.
Consider $f(t)=t^{2}$ and  $g(t)=2 t$. Then,
\begin{equation}
\label{eq:ex:cr:1}
(f\circ g)^{(\alpha)}(t) = \left(4 t^2\right)^{(\alpha)}
= 4 (\mu(t))^{1-\alpha} \left(\sigma(t) + t \right)
\end{equation}
while
\begin{equation}
\label{eq:ex:cr:2}
f^{(\alpha)}(g(t)) g^{(\alpha)}(t)
= (\mu(2t))^{1-\alpha} \left(\sigma(2t) + 2t\right) 2 (\mu(t))^{1-\alpha}
\end{equation}
and, for example for $\mathbb{T}=\mathbb{Z}$, it is easy to see that
$(f\circ g)^{(\alpha)}(t) \ne f^{(\alpha)}(g(t)) g^{(\alpha)}(t)$.
\end{example}

Note that when $\alpha = 1$ and $\mathbb{T} = \mathbb{R}$ our
derivative $f^{(\alpha)}$ reduces to the standard derivative $f'$
and, in this case, both expressions \eqref{eq:ex:cr:1}
and \eqref{eq:ex:cr:2} give $8 t$, as expected. In the fractional case $\alpha \in ]0,1[$
we are able to prove the following result, valid for an arbitrary time scale $\mathbb{T}$.

\begin{theorem}[Chain rule]
\label{T3}
Let $\alpha \in \left]0,1\right[$.
Assume $g:\mathbb{R}\rightarrow \mathbb{R}$ is continuous,
$g:\mathbb{T}\rightarrow \mathbb{R}$
is fractional differentiable of order $\alpha$ at $t \in \mathbb{T}^{\kappa}$,
and $f:\mathbb{R}\rightarrow \mathbb{R}$ is continuously differentiable.
Then there exists $c$ in the real interval $[t,\sigma(t)]$ with
\begin{equation}
\label{q1}
(f\circ g)^{(\alpha)}(t)=f'(g(c))g^{(\alpha)}(t).
\end{equation}
\end{theorem}

\begin{proof}
Let $t \in \mathbb{T}^{\kappa}$. First we consider
$t$ to be right-scattered. In this case
$$
(f\circ g)^{(\alpha)}(t)
=\frac{f(g(\sigma(t)))-f(g(t))}{(\mu(t))^{(\alpha)}}.
$$
If $g(\sigma(t))= g(t)$, then we get $(f\circ g)^{(\alpha)}(t)=0$
and $g^{(\alpha)}(t)=0$. Therefore, \eqref{q1} holds for any $c$
in the real interval $[t,\sigma(t)]$ and we can assume
$g(\sigma(t)) \neq g(t)$. By the mean value theorem,
\begin{equation*}
\begin{split}
(f\circ g)^{(\alpha)}(t)&=\frac{f(g(\sigma(t)))-f(g(t))}{g(\sigma(t))-g(t)}
\cdot \frac{g(\sigma(t))-g(t)}{(\mu(t))^{(\alpha)}}\\
&=f'(\xi)g^{(\alpha)}(t),
\end{split}
\end{equation*}
where $\xi$ is between $g(t)$ and $g(\sigma(t))$.
Since $g:\mathbb{R}\rightarrow \mathbb{R}$ is continuous, there is a
$c\in[t,\sigma(t)]$ such that $g(c)=\xi$, which gives us the desired result.
Now consider the case when $t$ is right-dense. In this case
\begin{equation*}
\begin{split}
(f\circ g)^{(\alpha)}(t)&=\lim_{s\rightarrow t}\frac{f(g(t))-f(g(s))}{g(t)-g(s)}
\cdot \frac{g(t)-g(s)}{(t-s)^{(\alpha)}}\\
&=\lim_{s\rightarrow t}\left\{f'(\xi_{s}).\frac{g(t)-g(s)}{(t-s)^{(\alpha)}}\right\}
\end{split}
\end{equation*}
by the mean value theorem, where $\xi_{s}$ is between $g(s)$ and $g(t)$.
By the continuity of $g$ we get that $\lim_{s\rightarrow t}\xi_{s}=g(t)$,
which gives us the desired result.
\end{proof}

\begin{example}
Let $\mathbb{T}=\mathbb{Z}$, for which $\sigma(t) = t+1$ and $\mu(t) \equiv 1$,
and consider the same functions of Example~\ref{ex:conterex:cr}:
$f(t)=t^{2}$ and $g(t)=2t$. We can find directly the value $c$,
guaranteed by Theorem~\ref{T3} in the interval $[4,\sigma(4)]=[4,5]$, so that
\begin{equation}
\label{eq:ex:cr:fc}
(f\circ g)^{(\alpha)}(4)=f'(g(c))g^{(\alpha)}(4).
\end{equation}
From \eqref{eq:ex:cr:1} it follows that $(f\circ g)^{(\alpha)}(4)=36$.
Because $g^{(\alpha)}(4)=2$ and $f'(g(c))=4c$, equality \eqref{eq:ex:cr:fc}
simplifies to $36 = 8 c$, and so $c=\frac{9}{2}$.
\end{example}

We end Section~\ref{sub:sec:FD} explaining how to compute fractional derivatives of higher-order.
As usual, we define the derivative of order zero as the identity operator: $f^{(0)} = f$.
\begin{definition}
\label{def:hofd}
Let $\beta$ be a nonnegative real number.
We define the fractional derivative of $f$ of order $\beta$ by
\begin{equation*}
f^{(\beta)}:=\left(f^{\Delta^{N}}\right)^{(\alpha)},
\end{equation*}
where $N := \lfloor \beta \rfloor$ (that is,
$N$ is the integer part of $\beta$) and $\alpha:=\beta - N$.
\end{definition}

Note that the $\alpha$ of Definition~\ref{def:hofd} is in the interval $[0,1]$.
We illustrate Definition~\ref{def:hofd} with some examples.

\begin{example}
If $f(t)=c$ for all $t\in \mathbb{T}$, $c$ a constant,
then $f^{(\beta)}\equiv 0$ for any $\beta \in\mathbb{R}_0^{+}$.
\end{example}

\begin{example}
Let $f(t) = t^2$, $\mathbb{T} = h \mathbb{Z}$, $h > 0$,
and $\beta=1.3$. Then, by Definition~\ref{def:hofd}, we have
$f^{(1.3)}=\left(f^{\Delta}\right)^{(0.3)}$. It follows
from $\sigma(t) = t+h$ that $f^{(1.3)}(t)=(2t + h)^{(0.3)}$.
Proposition~\ref{E1:i} and Theorem~\ref{T2} (i) and (ii) allow us to write that
$f^{(1.3)}(t)= 2 (t)^{(0.3)}$. We conclude from Proposition~\ref{E1:ii}
with $\mu(t) \equiv h$ that $f^{(1.3)}(t) = 2 h^{0.7}$.
\end{example}

% ---------------------------------------------

\subsection{Fractional Integration}
\label{sub:sec:FI}

The two major ingredients of any calculus
are differentiation and integration.
Now we introduce the fractional integral
on time scales.

\begin{definition}
\label{def:int}
Assume that $f:\mathbb{T}\rightarrow \mathbb{R}$ is a regulated function.
We define the indefinite fractional integral of $f$
of order $\beta$, $0 \leq \beta \leq 1$, by
\begin{equation*}
\int f(t)\Delta^{\beta}t := \left(\int f(t)\Delta t\right)^{(1-\beta)},
\end{equation*}
where $\int f(t)\Delta t$ is the usual indefinite integral of time scales \cite{BP}.
\end{definition}

\begin{remark}
It follows from Definition~\ref{def:int} that $\int f(t)\Delta^{1}t = \int f(t)\Delta t$
and $\int f(t)\Delta^{0}t = f(t)$.
\end{remark}

\begin{definition}
\label{def:intFracCauchy}
Assume $f:\mathbb{T}\rightarrow \mathbb{R}$ is a regulated function. Let
$$
F^{\beta}(t)=\int f(t)\Delta^{\beta} t
$$
denote the indefinite fractional integral
of $f$ of order $\beta$ with $0 \leq \beta \leq 1$.
We define the Cauchy fractional integral by
\begin{equation*}
\int_{a}^{b}f(t)\Delta^{\beta} t := \left. F^{\beta}(t)\right|^b_a
=F^{\beta}(b)-F^{\beta}(a), \quad a,b\in \mathbb{T}.
\end{equation*}
\end{definition}

The next theorem gives some properties of the fractional integral of order $\beta$.

\begin{theorem}
\label{T4}
If $a, b, c \in \mathbb{T}$, $\xi\in\mathbb{R}$,
and $f,g\in \mathcal{C}_{rd}$ with $0\leq \beta\leq 1$, then
\begin{description}
\item[(i)] $\int_{a}^{b}[f(t)+g(t)]\Delta^{\beta} t
= \int_{a}^{b}f(t)\Delta^{\beta} t + \int_{a}^{b}g(t)\Delta^{\beta} t$;

\item[(ii)] $\int_{a}^{b}(\xi f)(t)\Delta^{\beta} t = \xi
\int_{a}^{b}f(t)\Delta^{\beta} t$;

\item[(iii)] $\int_{a}^{b}f(t)\Delta^{\beta} t
= - \int_{b}^{a}f(t)\Delta^{\beta} t$;

\item[(iv)] $\int_{a}^{b}f(t)\Delta^{\beta} t
= \int_{a}^{c}f(t)\Delta^{\beta} t + \int_{c}^{b}f(t)\Delta^{\beta} t$;

\item[(v)] $\int_{a}^{a}f(t)\Delta^{\beta} t = 0$.
\end{description}
\end{theorem}

\begin{proof}
The equalities follow from Definition~\ref{def:int} and Definition~\ref{def:intFracCauchy},
analogous properties of the delta integral of time scales, and the properties
of Section~\ref{sub:sec:FD} for the fractional derivative on time scales.
$(i)$ From Definition~\ref{def:intFracCauchy}
\begin{equation*}
\int_{a}^{b}(f+g)(t)\Delta^{\beta} t
= \left. \int \left(f(t) + g(t)\right) \Delta^{\beta} t \right|_a^b
\end{equation*}
and, from Definition~\ref{def:int},
\begin{equation*}
\int_{a}^{b}(f+g)(t)\Delta^{\beta} t
= \left. \left(\int \left(f(t) + g(t)\right) \Delta t\right)^{(1-\beta)} \right|_a^b.
\end{equation*}
It follows from the properties of the delta integral and Theorem~\ref{T2} (i) that
\begin{equation*}
\int_{a}^{b}(f+g)(t)\Delta^{\beta} t
= \left. \left(\int f(t) \Delta t\right)^{(1-\beta)}
+ \left(\int g(t) \Delta t\right)^{(1-\beta)}\right|_a^b.
\end{equation*}
Using again Definition~\ref{def:int} and Definition~\ref{def:intFracCauchy},
we arrive to the intended relation:
\begin{equation*}
\begin{split}
\int_{a}^{b}(f+g)(t)\Delta^{\beta} t
&= \left. \int f(t) \Delta^\beta t + \int g(t) \Delta^\beta t\right|_a^b\\
&= \left. F^\beta(t) + G^\beta(t)\right|_a^b
= F^\beta(b) + G^\beta(b) - F^\beta(a) - G^\beta(a)\\
&= \int_{a}^{b}f(t)\Delta^{\beta} t + \int_{a}^{b}g(t)\Delta^{\beta} t.
\end{split}
\end{equation*}
$(ii)$ From Definition~\ref{def:intFracCauchy} and Definition~\ref{def:int} one has
\begin{equation*}
\int_{a}^{b}(\xi f)(t)\Delta^{\beta} t
=\left. \int (\xi f)(t)\Delta^\beta t\right|_a^b
=\left. \left(\int (\xi f)(t)\Delta t\right)^{(1-\beta)}\right|_a^b.
\end{equation*}
It follows from the properties of the delta integral and Theorem~\ref{T2} (ii) that
\begin{equation*}
\int_{a}^{b}(\xi f)(t)\Delta^{\beta} t
= \left. \xi \left(\int f(t)\Delta t\right)^{(1-\beta)}\right|_a^b.
\end{equation*}
We conclude the proof of (ii) by using again Definition~\ref{def:int}
and Definition~\ref{def:intFracCauchy}:
\begin{equation*}
\begin{split}
\int_{a}^{b}(\xi f)(t)\Delta^{\beta} t
&= \left. \xi \int f(t)\Delta^\beta t\right|_a^b
= \left. \xi F^\beta(t)\right|_a^b
= \xi\left(F^\beta(b)-F^\beta(a)\right)\\
&= \xi \int_a^b f(t) \Delta^\beta t.
\end{split}
\end{equation*}
The last three properties are direct consequences of Definition~\ref{def:intFracCauchy}:\\
$(iii)$
\begin{equation*}
\begin{split}
\int_{a}^{b}f(t)\Delta^{\beta} t
&= F^\beta(b) - F^\beta(a)
= - \left(F^\beta(a)-F^\beta(b)\right)\\
&= -\int_{b}^{a}f(t)\Delta^{\beta} t.
\end{split}
\end{equation*}
$(iv)$
\begin{equation*}
\begin{split}
\int_{a}^{b}f(t)\Delta^{\beta} t
&= F^\beta(b) - F^\beta(a)
= F^\beta(c) - F^\beta(a) + F^\beta(b) - F^\beta(c)\\
&=\int_{a}^{c} f(t)\Delta^{\beta} t + \int_{c}^{b} f(t)\Delta^{\beta} t.
\end{split}
\end{equation*}
$(v)$
\begin{equation*}
\int_{a}^{a}f(t)\Delta^{\beta} t = F^\beta(a) - F^\beta(a) = 0.
\end{equation*}
The proof is complete.
\end{proof}

We end with a simple example of a discrete fractional integral.

\begin{example}
Let $\mathbb{T} = \mathbb{Z}$, $ 0 \le \beta \le 1$, and $f(t) = t$.
Using the fact that in this case
$$
\int t \Delta t = \frac{t^2}{2} + C
$$
with $C$ a constant, we have
$$
\int_{1}^{10} t \, \Delta^\beta t
= \left. \int t \, \Delta^\beta t \right|_{1}^{10}
= \left. \left(\int t \, \Delta t\right)^{(1-\beta)} \right|_{1}^{10}
= \left. \left(\frac{t^2}{2} + C\right)^{(1-\beta)} \right|_{1}^{10}.
$$
It follows from Example~\ref{ex:17} (i) with $\mu(t) \equiv 1$,
Theorem~\ref{T2} (i) and (ii) and Proposition~\ref{E1:i} that
$$
\int_{1}^{10} t \, \Delta^\beta t
= \left. \frac{1}{2} \left(2 t + 1\right) \right|_{1}^{10}
= \frac{21}{2} - \frac{3}{2} = 9.
$$
\end{example}

% ------------------------

\section{Conclusion}
\label{sec:Conc}

Fractional calculus, that is, the study of differentiation and integration
of noninteger order, is here extended, via the recent and powerful calculus on time scales,
to include, in a single theory, the discrete fractional difference calculus
and the local continuous fractional differential calculus.
We have only introduced some fundamental concepts
and proved some basic properties, and much remains
to be done in order to develop the theory here initiated:
to prove concatenation properties of derivatives and integrals,
to consider partial fractional operators on time scales,
to introduce a suitable fractional exponential on time scales,
to study boundary value problems for fractional
differential equations on time scales,
to investigate the usefulness of the new fractional calculus
in applications to real world problems where the time scale is partially
continuous and partially discrete with a time-varying graininess function, etc.
We would like also to mention that it is possible
to develop fractional calculi on time scales
in other different directions than the one considered here.
For instance, instead of following the delta approach we have adopted,
one can develop a nabla \cite{Alm:Tor:JVC,naty:NA:2009},
a diamond \cite{Mal:Tor:diamond,Moz}, or a symmetric \cite{MyID:246,MyID:247}
time scale fractional calculus. These and other questions
will be subject of future research.

% ------------------------

\section*{Acknowledgments}

This research was initiated while N. Benkhettou was visiting
the Department of Mathematics of University of Aveiro, February and March of 2013.
The hospitality of the host institution and the financial
support of Sidi Bel Abbes University are here gratefully
acknowledged. A. M. C. Brito da Cruz and D. F. M. Torres were supported
by Portuguese funds through the
\emph{Center for Research and Development in Mathematics and Applications} (CIDMA)
and \emph{The Portuguese Foundation for Science and Technology} (FCT),
within project PEst-OE/MAT/UI4106/2014.
The authors are very grateful to three referees
for valuable remarks and comments, which
significantly contributed to the quality of the paper.

% ------------------------

\small

% ------------------------

% ------------------------

\end{document}